\documentclass[12pt]{article}
\usepackage[utf8]{inputenc}
\usepackage{graphicx}
\usepackage{genyoungtabtikz}
\usepackage{ytableau}
\usepackage{amsmath}
\usepackage{amssymb}
\usepackage{amsthm}
\usepackage{cite}
\usepackage{mathtools}
\usepackage{fullpage}
\usepackage{authblk}
\usepackage{hyperref}

\begin{document}

\theoremstyle{plain}
\newtheorem{theorem}{Theorem}
\newtheorem{corollary}[theorem]{Corollary}
\newtheorem{lemma}[theorem]{Lemma}
\newtheorem{proposition}[theorem]{Proposition}
\theoremstyle{definition}
\newtheorem{definition}[theorem]{Definition}
\newtheorem{example}[theorem]{Example}
\newtheorem{conjecture}[theorem]{Conjecture}
\theoremstyle{remark}
\newtheorem{remark}[theorem]{Remark}

\newcommand{\seqnum}[1]{\href{https://oeis.org/#1}{\rm \underline{#1}}}

\begin{center}
\vskip 1cm{\LARGE\bf 
The Gini Index of an Integer Partition}
\vskip 1cm
\large
Grant Kopitzke \\
Department of Mathematics\\
         University of Wisconsin, Milwaukee\\
         Milwaukee, WI 53211\\
	 USA \\
\href{mailto:kopitzk3@uwm.edu}{\tt kopitzk3@uwm.edu} \\
\end{center}

\vskip .2 in

\begin{abstract}
The Gini index is a number that attempts to measure how equitably a resource is distributed throughout a population, and is commonly used in economics as a measurement of inequality of wealth or income. The Gini index is often defined as the area between the Lorenz curve of a distribution and the line of equality, normalized to be between zero and one. In this fashion, we define a Gini index on the set of integer partitions and show that it is closely related to the second elementary symmetric polynomial, and the dominance order on partitions. We conclude with a generating function for the Gini index, and discuss how it can be used to find lower bounds on the width of the dominance lattice. 
\end{abstract}

\section{Introduction}
\label{intro}
In part one of his 1912 book ``Variabilit\`a e Mutabilit\`a" (Variability and Mutability), the statistician Corrado Gini formulated a number of different summary statistics; among which was what is now known as the \textit{Gini index} --- a measure that quantifies how equitably a resource is distributed throughout a population. Referring to ``the" Gini index can be misleading, as no fewer than thirteen formulations of his famous index appeared in the original publication \cite{Origins:1}. Since then, many others have appeared in a variety of different fields.

The Gini index is usually defined using the \textit{Lorenz curve}. In ``Methods of Measuring the Concentration of Wealth", Lorenz defined this curve in the following fashion. Consider a population of people amongst whom is distributed some fixed amount of wealth.
Let ${L(x)}$ be the percentage of total wealth possessed by the poorest $x$ percent of the population. The graph ${y=L(x)}$ is the Lorenz curve of the population \cite{Lorenz:1}.

It is clear from this definition that $L(0)=0$ (i.e., none of the people have none of the wealth), ${L(1)=1}$ (all of the people have all of the wealth), and $L$ is non-decreasing. Since any population of people must have finite size $n$, the function ${L(x)}$ as defined above would appear to be a discrete function on the set ${\{\frac{k}{n}:k\in \mathbb{Z}\text{ and }0\leq k \leq n\}}$. However, in practice $L$ is often made continuous on ${[0,1]}$ by linear interpolation \cite{Farris:1}.

If each person possesses the same amount of wealth, then the Lorenz curve for this distribution is the line $y=x$, which we call the ``line of equality". The area between the line of equality and the Lorenz curve of a wealth distribution provides a measurement of the wealth inequality in that population. \begin{figure}[htp]
    \centering
    \includegraphics[width=10cm]{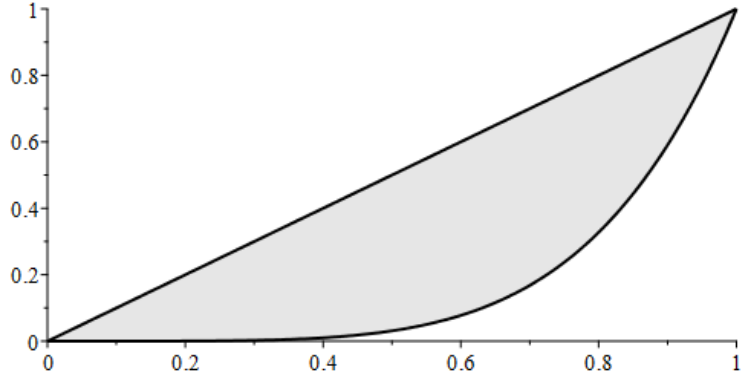}
    \caption{Area between the line of equality and a typical Lorenz curve}
\end{figure}

The maximum possible area of ${\frac{1}{2}}$ arises from the distribution in which one person controls all of the wealth (${L(1)=1}$, and ${L(x)=0}$ for all ${x\neq1}$). The Gini index of a distribution is then defined by calculating the area between the line of equality and Lorenz curve of the distribution, and normalizing this area to be between zero and one.
\[G=2\int_{0}^{1}{\big(x-L(x)\big)dx}\]

In this paper we consider distributions of a discrete indivisible resource in a finite population, where the amount of that resource is equal to the number of people in the population. There is a natural one-to-one correspondence between the set of such distributions with $n$ people, and the set of partitions of $n$. We will then define the Gini index of a integer partition in a similar fashion as above; for a comparable Gini index defined on a partition of a set, see \cite{Simovici:1}.

\section{Preliminaries}
\label{prelim}
\subsection{Partitions and Young diagrams}
A \textit{partition}, $\lambda$, of a positive integer $n$ (sometimes written as ${\lambda \vdash n}$) is a sequence ${(\lambda_1 , \lambda_2 ,\ldots, \lambda_{\ell})}$ of ${\ell\leq n}$ non-increasing positive integers such that ${\sum_{i=1}^{\ell}{\lambda_{i}}=n}$. The ${\lambda_i}$ (${1\leq i \leq \ell}$) are called the ``parts of $\lambda$''. To avoid repeating parts, it is sometimes useful to write a partition as $(\lambda_1^{a_1},\lambda_2^{a_2},\ldots,\lambda_{\ell}^{a_{\ell}})$ to represent $\lambda_i$ repeating $a_i$ times. In this case, we have that $\sum_{i=1}^{\ell}{{a_i}\lambda_i}=n$, and $\lambda_i\neq \lambda_j$ for all $i \neq j$. This notation will be used in the proof of Proposition~\ref{prop3}. In order to make the length of $\lambda$ (the number of parts) equal to $n$, one can ``pad out'' the partition by adding ${n-\ell}$ zeros to the end. For example, the partition ${(4,3,1,1)}$ of 9 is equivalent to ${(4,3,1,1,0,0,0,0,0).}$ This technique will be used when defining the Lorenz curve of a partition.

A \textit{Young diagram} is a finite collection of boxes arranged in left-justified rows, with a weakly decreasing number of boxes in each row \cite{Fulton:1}. Integer partitions are in one-to-one correspondence with Young diagrams in the following way: if ${\lambda = (\lambda_1 , \lambda_2 ,\ldots, \lambda_{\ell})}$ is a partition of $n$ then the Young diagram of shape ${\lambda}$ has ${\lambda_1}$ boxes in its first row, ${\lambda_2}$ boxes in its second row, etc. For example, if ${\lambda = (4,3,1,1)}$, then the Young diagram of shape $\lambda$ is
\[\yng(4,3,1,1)\,.\]

The \textit{conjugate partition} $\widetilde{\lambda}$ of $\lambda$ is the partition of $n$ obtained by reflecting the Young diagram of $\lambda$ across its main diagonal. As in the previous example, if ${\lambda = (4,3,1,1)}$, then the Young diagram of $\widetilde{\lambda}$ is 
\[\yng(4,2,2,1)\,,\]
hence ${\widetilde{\lambda}=(4,2,2,1)}$. Conjugation is clearly a bijection on the partitions of $n$.

The \textit{dominance order} is a partial order on the set of partitions of $n$.  If $\lambda=(\lambda_1,\lambda_2,\ldots,\lambda_n)$ and $\mu=(\mu_1,\mu_2,\ldots,\mu_n)$ are partitions of $n$, then $\mu \preceq \lambda$ if 
\[\sum_{i=1}^{k}{\mu_i} \leq \sum_{i=1}^{k}{\lambda_i} \]
for all $k\geq 1$. Conjugation of partitions is an antiautomorphism on the dominance lattice of partitions of $n$ \cite{Brylawski:1}. In other words, if $\mu \preceq \lambda$, then $\widetilde{\lambda}\preceq\widetilde{\mu}$. We will write $\mu \prec \lambda$ if $\mu \preceq \lambda$ and $\mu \neq \lambda$, and will let $P_n$ denote the partially ordered set of partitions of $n$ with respect to dominance.

For a fixed positive integer $n$, an $\textit{antichain}$ in $P_n$ is a subset of $P_n$ in which all partitions are pairwise incomparable. A $\textit{maximum antichain}$ is an antichain of maximal cardinality. The length of the maximum antichain is also known as the $\textit{width of the lattice}$. The width of $P_n$ \seqnum{A076269} is currently an open problem.

\subsection{The second elementary symmetric polynomial}
The \textit{second elementary symmetric polynomial}, $e_2$, in $n$ variables, ${x_1,\,x_2,\ldots\,x_n}$, is defined \[e_2(x_1,x_2,\ldots,x_n)=\sum_{1\leq i < j \leq n}{x_ix_j}.\] 
For example, if ${\lambda=(4,3,1,1)}$ is a partition of $9$, then 
\[e_2(\lambda)=\Big(4(3+1+1)+3(1+1)+1(1) \Big)=27.\]
We will make use of the following result.

\begin{lemma}
If ${\lambda=(\lambda_1, \lambda_2 ,\ldots, \lambda_{\ell})}$ is a partition of a positive integer $n$, then 
\[e_2(\lambda)=\binom{n+1}{2}-\sum_{i=1}^{\ell}\binom{\lambda_i+1}{2}.\]

\label{lemma1}
\end{lemma}

\begin{proof}
Let ${\lambda=(\lambda_1,\lambda_2,\ldots,\lambda_{\ell})}$ be a partition of $n$. Note that ${\sum_{i=1}^{\ell}{\lambda_i}=n}$. Then 
\begin{align*}
    e_2(\lambda) &= \sum_{1\leq i < j \leq \ell}{\lambda_i \lambda_j}\\
    &=\binom{n+1}{2}-\left( \sum_{1\leq i < j \leq \ell}{\big(-\lambda_i \lambda_j \big)}+\binom{n+1}{2}\right)\\
    &=\binom{n+1}{2}-\frac{1}{2}\left( \sum_{1\leq i < j \leq \ell}{\big(-2\lambda_i \lambda_j \big)}+n(n+1)\right)\\
    &=\binom{n+1}{2}-\frac{1}{2}\left( \sum_{1\leq i < j \leq \ell}{\big(-2\lambda_i \lambda_j \big)}+\left(\sum_{i=1}^{\ell}{\lambda_i}\right)\left(\sum_{j=1}^{\ell}{\lambda_j}+1\right)\right)\\
    &=\binom{n+1}{2}-\frac{1}{2}\left( \sum_{1\leq i < j \leq \ell}{\big(-2\lambda_i \lambda_j \big)}+\sum_{i=1}^{\ell}{\left(\lambda_{i}^{2}+\lambda_i\right)} + \sum_{1\leq i<j\leq \ell}{\big(2\lambda_i \lambda_j\big)} \right)\\
    &=\binom{n+1}{2}-\frac{1}{2}\left( \sum_{i=1}^{\ell}{\lambda_{i}(\lambda_i+1)} \right)\\
    &=\binom{n+1}{2}-\left( \sum_{i=1}^{\ell}{\binom{\lambda_i+1}{2}} \right).\\
\end{align*}
\end{proof}

\newpage
\section{The Gini index of an integer partition}
\label{gini}
As previously stated, we restrict our study of the Gini index to finite populations where the amount of a discrete indivisible resource is equal to the size of the population. In other words, there is one of said resource available for each person. The distributions of $n$ of such a resource amongst $n$ people is in one-to-one correspondence with the integer partitions of $n$.

For example, if there are $4$ dollars in a population of $4$ people, then the partition ${(3,1)}$ of $4$ would correspond to one person having $3$ dollars, one person having $1$ dollar, and the two remaining people having nothing. Whereas the partitions ${(1,1,1,1)}$ and ${(4)}$ correspond to completely equitable and completely inequitable distributions, respectively.

Given a partition $\lambda=(\lambda_1,\ldots,\lambda_n)$ of a positive integer $n$ (padded with zeros on the tail, if necessary), the $\textit{Lorenz curve of $\lambda$}$, $L_{\lambda}:[0,n]\longrightarrow [0,n]$, is defined as $L_{\lambda}(0)=0$, and $L_{\lambda}(x)=\sum_{i=n-k+1}^{n}{\lambda_i}$, where $1\leq k\leq n$ is the unique positive integer such that $x\in(k-1,k]$.
In other words, for $k$ from $1$ to $n$, the Lorenz curve of $\lambda$ on the interval $(k-1,k]$ is the sum of the last $k$ parts of $\lambda$, i.e.,  $\lambda_n+\lambda_{n-1}+\dots+\lambda_{n-k+1}$. Since total equality corresponds to the flat partition $(1^n)$, using the above definition for the Lorenz curve of a partition, we find that the line of equality is given by $y=\lceil x \rceil$.

\begin{figure}[htp]
    \centering
    \includegraphics[width=10cm]{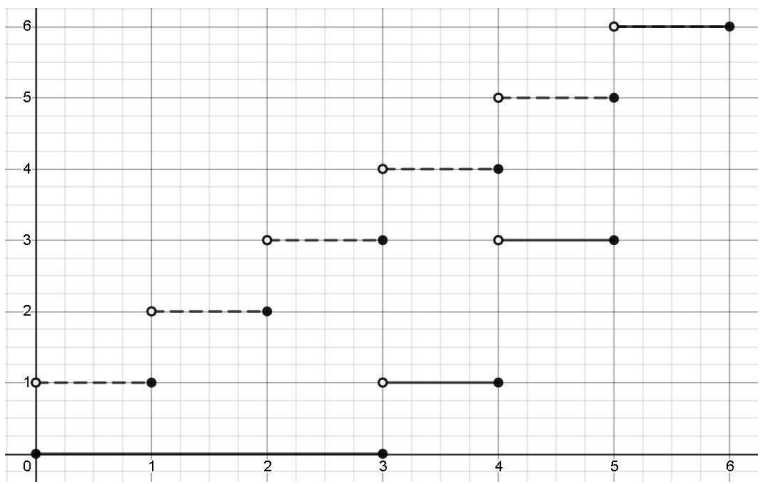}
    \caption{The line of equity (dashed) and the Lorenz curve of the partition (3,2,1) of 6 (solid).}
\end{figure}

The standard Gini index is calculated by finding the area between the line of equality and the Lorenz curve, and normalizing. In a similar fashion, we define the Gini index, $g$, of a partition $\lambda=(\lambda_1,\ldots,\lambda_n)$ of $n$ by 
\begin{align*}
    g(\lambda)&= \int_{0}^{n} \left( \lceil x \rceil -L_{\lambda}(x) \right) dx\\
    &=\binom{n+1}{2}-\sum_{i=1}^{n}{i\lambda_i}.
\end{align*}

The ordinary Gini index is normalized to be between zero and one. For a fixed value of $n$, the function $g$ attains its maximum value of $\binom{n}{2}$ on the partition $(n)$ of $n$. So the Gini index of a partition $\lambda$ of $n$ can be normalized by dividing $g(\lambda)$ by $\binom{n}{2}$. As long as $n$, and ${g(\lambda)}$ are both known, the normalized Gini index of $\lambda$ can always be calculated in this fashion. With this in mind, we may disregard the normalization, and view $g$ itself as the integer valued Gini index of a partition.

Our construction of $g$ in conjunction with Lemma~\ref{lemma1} yields some interesting results:

\begin{proposition} 
If $\lambda$ is an integer partition, then ${g(\lambda)=e_{2}(\widetilde{\lambda})}$, where ${\widetilde{\lambda}}$ is the conjugate partition of $\lambda$.
\label{prop1}
\end{proposition}

\begin{proposition}
Let $\lambda$ and $\mu$ be partitions of $n$. If $\mu \prec \lambda$, then $g(\mu)<g(\lambda)$ and $e_{2}(\lambda)<e_{2}(\mu)$.
\label{prop2}
\end{proposition}
The normalized Gini index on $\mathbb{R}^n$ restricted to $P_n$ is equal to $\frac{2g}{n^2}$. Arnold et.al. \cite{Inequalities:1} have shown that the normalized Gini index is strictly Schur convex, so Proposition~\ref{prop2} follows from this fact. A complete proof of Proposition~\ref{prop2} that does not utilize these facts will be given in section 5.

The converse of Proposition~\ref{prop2} does not hold, in general. However, the contrapositive provides us with an easily calculated lower bound on the width of $P_n$. For if $\lambda,\mu\in P_n$ are distinct partitions such that $g(\lambda)= g(\mu)$, then the partitions $\lambda$ and $\mu$ are incomparable. Such lower bounds can be calculated using the generating function of $g$ given in the following section.

\section{Generating functions}
\label{gen}

It is often useful in Algebraic Combinatorics to record a discrete data set in the coefficients or powers of a formal power series. We call these power series ``generating functions'' for the data set. By ``formal'' we mean that the convergence of the series is immaterial. Any variables appearing in the series are taken as indeterminates rather than numbers. Alternatively, one may consider a formal power series as an ordinary power series that converges only at zero.

We define a generating function  for the Gini index $g(\lambda)$ of an integer partition $\lambda$ by
\[G(q,x)=\sum_{n=1}^{\infty}{\sum_{\lambda \vdash n}{q^{\left(\binom{n+1}{2}-g(\lambda)\right)}x^n}}.\]

Perhaps the most widely known example of a generating function is that of the integer partition function ${P(n)}$, which counts the number of partitions of the integer $n$. For example, ${n=4}$ has partitions

\begin{center}
    $(1,1,1,1)$,
    $(2,1,1)$,
    $(2,2)$,
    $(3,1)$, and
    $(4)$,
\end{center}
so ${P(4)=5.}$ It is well known \cite{Andrews:1} that ${P(n)}$ has generating function
\[\prod_{n=1}^{\infty}{\frac{1}{1-x^n}}=\sum_{n=0}^{\infty}{P(n)x^n},\]
where $P(0)$ is defined to be $1$.

In light of our previous results, we obtain a similar equality for ${G(q,x)}$.
\begin{proposition}
\[\prod_{n=1}^{\infty}{\frac{1}{1-q^{\binom{n+1}{2}}x^n}}-1=\sum_{n=1}^{\infty}{\sum_{\lambda \vdash n}{q^{\left(\binom{n+1}{2}-g(\lambda)\right)}x^n}}.\]
\label{prop3}
\end{proposition}
The proof is provided in the following section.

We can use $G(q,x)$ to find lower bounds on the width of $P_n$ by calculating the cardinalities of the maximum level sets of $g$ on $P_n$. In particular, the cardinality of the maximum level set of $g$ on $P_n$ will be the largest coefficient on the powers of $q$ that form the coefficient of $x^n$. Expanding $G(q,x)$ yields
\begin{align*}
        \sum_{n=1}^{\infty}\sum_{\lambda \vdash n}q^{\left(\binom{n+1}{2}-g(\lambda)\right)}x^{n}
        &=qx+(q^2+q^3)x^2+(q^3+q^4+q^6)x^3\\
        &+(q^4+q^5+q^6+q^7+q^{10})x^4\\
        &+(q^5+q^6+q^7+q^8+q^9+q^{11}+q^{15})x^5\\
        &+(\cdots+{2}q^9+\cdots)x^6+(\cdots+2q^{10}+\cdots)x^7\\
        &+(\cdots+2q^{11}+\cdots)x^8+(\cdots+3q^{15}+\cdots)x^9\\
        &+\cdots.
    \end{align*}
So the cardinalities of the maximum level sets of $g$ are 1, 1, 1, 1, 1, 2, 2, 2, and 3, on $P_1$ through $P_9$, respectively.

Let $b(n)$ denote the cardinality of the maximal level set of $g$ on $P_n$ \seqnum{A337206}. Proposition~\ref{prop2} implies that $b(n)\leq a(n)$ for all positive integers $n$, where $a(n)$ is the size of the maximum antichain in $P_n$ \seqnum{A076269}. Early \cite{Antichain:1} showed that ${\Omega\left(n^{-5/2}e^{\pi \sqrt{2n/3}}\right)\leq a(n)}$. It is currently unknown how $b(n)$ relates asymptotically to $a(n)$ or $n^{-5/2}e^{\pi \sqrt{2n/3}}$. However, we conjecture that ${b(n)\leq O\left(n^{-5/2}e^{\pi \sqrt{2n/3}} \right)}.$

\section{Proofs of main results}
\subsection{Proof of Proposition~\ref{prop1}}
\begin{proof}
Let ${\lambda=(\lambda_1,\lambda_2,\ldots,\lambda_{\ell})}$ be a partition of a positive integer $n$, where
\[\lambda_1\geq\lambda_2\geq\ldots\geq\lambda_{\ell}>0\]
and ${\sum_{i=1}^{\ell}{\lambda_i}=n}$. We can calculate $g{(\lambda)}$ by filling the Young diagram of shape $\lambda$ with numbers, where the entry in any box counts the number of boxes in that column that are weakly above it. For example, for the partition ${(4,3,1,1)}$, we would have
\[\young(1111,222,3,4)\,.\]
Then the sums of the values in each row are
\begin{align*}
    \sum{\big(\textrm{Entries in row }1\big)}&=\lambda_1,\\
    \sum{\big(\textrm{Entries in row }2\big)}&=2\lambda_2,\\
    \sum{\big(\textrm{Entries in row }3\big)}&=3\lambda_3,\\
    &\shortvdotswithin{=}
    \sum{\big(\textrm{Entries in row }\ell\big)}&=\ell\lambda_{\ell}.\\
\end{align*}
Summing all values in the Young diagram of $\lambda$ yields $\sum_{i=1}^{\ell}i\lambda_{i}$. By subtracting this from ${\binom{n+1}{2}}$ we have
\[\binom{n+1}{2}-\sum{\big(\textrm{Entries in Young Diagram }i\big)}=\binom{n+1}{2}-\sum_{i=1}^{\ell}{i\lambda_i}
    =g(\lambda).\]

We can calculate ${e_2(\lambda)}$ similarly by forming a Young diagram of shape $\lambda$ where each box's entry counts the number of boxes in the same row that are weakly to the left of its own. Again using ${(4,3,1,1)}$ as an example, we would have
\[\young(1234,123,1,1)\,.\]
In general, the ${i^{\textrm{th}}}$ row  of the diagram for $\lambda$ will be of the form 

\[\ytableausetup{mathmode, boxsize=2em}
\begin{ytableau}
\scriptstyle{1} & \scriptstyle{2} & \none[\dots] & \scriptstyle{\lambda_i-1} & \scriptstyle{\lambda_i} \\
\end{ytableau}\,,\]
so the sum of the boxes in the ${i^{\textrm{th}}}$ row will be ${\binom{\lambda_i+1}{2}}$. Summing all of the entries in the Young diagram of $\lambda$ and subtracting this from ${\binom{n+1}{2}}$ yields
\[\binom{n+1}{2}-\sum_{i=1}^{\ell}{\big(\textrm{Entries in row }i\big)}=\binom{n+1}{2}-\sum_{i=1}^{\ell}\binom{\lambda_i+1}{2}=e_2(\lambda),\]
where the last equality is by Lemma~\ref{lemma1}. Since ${g(\lambda)}$ is calculated by counting boxes in the columns of the Young diagram of $\lambda$, and ${e_2(\lambda)}$ is calculated by counting boxes in the rows, it follows that $g(\lambda)=e_2(\widetilde{\lambda})$.

\end{proof}
\subsection{Proof of Proposition~\ref{prop2}}
\begin{proof}
Let $\lambda=(\lambda_1,\lambda_2,\ldots,\lambda_n)$ and $\mu=(\mu_1,\mu_2,\ldots,\mu_n)$ be partitions of $n$ (padded with zeros in their tails, if necessary). Suppose that $\lambda$ covers $\mu$, i.e., there is no partition $\rho$ of $n$ such that $\mu \prec \rho \prec \lambda$. Now $\lambda$ covers $\mu$ if and only if 
\begin{align*}
    \lambda_i &= \mu_i +1,\\
    \lambda_k &= \mu_k -1, and\\
    \lambda_j &= \mu_j,
\end{align*}
for all $j \neq i$ or $k$, where $k=i+1$ or $\mu_i=\mu_k$ \cite{Brylawski:1}. In other words, $\lambda$ covers $\mu$ if and only if all but two of the rows (row $i$ and $k$, with $i<k$) in the Young diagrams of $\lambda$ and $\mu$ are of the same length, and the diagram of $\lambda$ can be obtained from that of $\mu$ by removing the last box from the $k^{\textrm{th}}$ row, and appending it to the end of the $i^{\textrm{th}}$ row.

Begin with the Young diagram of $\mu$ and, as in the proof of Proposition~\ref{prop1}, fill the diagram with numbers so that each box's entry counts the number of boxes weakly to the left of it.
\[\ytableausetup
{mathmode, boxsize=2em}
\begin{ytableau}
\scriptstyle{1} & \scriptstyle{2} & \none[\dots] & \scriptstyle{\mu_1-3} & \scriptstyle{\mu_1-2} & \scriptstyle{\mu_1-1} & \scriptstyle{\mu_1} \\
\none[\vdots]\\
\scriptstyle{1} & \scriptstyle{2} & \none[\dots] & \scriptstyle{\mu_i-1} & \scriptstyle{\mu_i} \\
\none[\vdots]\\
\scriptstyle{1} & \scriptstyle{2} & \none[\dots] & \scriptstyle{\mu_k-1} & \scriptstyle{\mu_k} \\
\none[\vdots]\\
\scriptstyle{1} & \scriptstyle{2} & \none[\dots] & \scriptstyle{\mu_n} \\
\end{ytableau}\]
From row $k$ we remove the box containing $\mu_k$ and append it to the end of row $i$ to obtain a diagram of shape $\lambda$.
\[\ytableausetup
{mathmode, boxsize=2em}
\begin{ytableau}
\scriptstyle{1} & \scriptstyle{2} & \none[\dots] & \scriptstyle{\mu_1-3} & \scriptstyle{\mu_1-2} & \scriptstyle{\mu_1-1} & \scriptstyle{\mu_1} \\
\none[\vdots]\\
\scriptstyle{1} & \scriptstyle{2} & \none[\dots] & \scriptstyle{\mu_i-1} & \scriptstyle{\mu_i} & \scriptstyle{\mu_k}\\
\none[\vdots]\\
\scriptstyle{1} & \scriptstyle{2} & \none[\dots] & \scriptstyle{\mu_k-1}  \\
\none[\vdots]\\
\scriptstyle{1} & \scriptstyle{2} & \none[\dots] & \scriptstyle{\mu_n} \\
\end{ytableau}\]
But $i < k$, hence $\mu_k \leq \mu_i$, and the corresponding filling of the diagram for $\lambda$ would have the last cell in row $i$ containing $\mu_i+1$, which is strictly greater than $\mu_k$. Thus the sum of all numbers in the diagram for $\lambda$ is 
\[\sum_{\substack{j=1 \\ j\neq i,k}}^{n}{\binom{\mu_j+1}{2}}+{\binom{\mu_i+2}{2}} + {\binom{\mu_k}{2}}, \]
and the sum of all numbers in the diagram for $\mu$ is
\[\sum_{j=1}^{n}{\binom{\mu_j+1}{2}}. \]

By Lemma~\ref{lemma1}, we have
\begin{align*}
    e_2(\mu)-e_2(\lambda) =& \,\binom{n+1}{2}-\sum_{j=1}^{n}{\binom{\mu_j+1}{2}}-\\
    &\,\left( \binom{n+1}{2}-\left(\sum_{\substack{j=1 \\ j\neq i,k}}^{n}{\binom{\mu_j+1}{2}}+{\binom{\mu_i+2}{2}} + {\binom{\mu_k}{2}} \right)\right)\\
    =&\,{\binom{\mu_i+2}{2}}+{\binom{\mu_k}{2}}-{\binom{\mu_i+1}{2}}-{\binom{\mu_k+1}{2}}\\
    =&\,\frac{(\mu_i+1)(\mu_i+2-\mu_i)+(\mu_k)(\mu_k-2-\mu_k)}{2}\\
    =&\,\mu_i+1-\mu_k\\
    >&\,0.
\end{align*}
So $e_2(\mu)>e_2(\lambda)$. Moreover $\mu \prec \lambda$ if and only if $\widetilde{\lambda} \prec \widetilde{\mu}$. Hence $\widetilde{\mu}$ covers $\widetilde{\lambda}$, and by Proposition~\ref{prop1}, $g(\mu)< g(\lambda).$ The general case follows by transitivity.
\end{proof}

\subsection{Proof of Proposition~\ref{prop3}}
\begin{proof}
We will show that the power series about $x=0$ of the product \begin{equation}
    {\prod_{n=1}^{\infty}{\frac{1}{1-q^{\binom{n+1}{2}}x^n}}}-1
    \label{eq1}
\end{equation} has as its general coefficient \[{\sum_{\lambda\vdash n}{q^{\left(\binom{n+1}{2}-g(\lambda)\right)}}}.\]
Considering each factor of the product as a geometric series, we have

\begin{align*}
    \prod_{n=1}^{\infty}{\frac{1}{1-q^{\binom{n+1}{2}}x^{n}}}=
    &
    \frac{1}{\left(1-q^{\binom{2}{2}}x\right)}
    \cdot
    \frac{1}{\left(1-q^{\binom{3}{2}}x^2\right)}
    \cdot
    \frac{1}{\left(1-q^{\binom{4}{2}}x^3\right)}
    \cdot
    \frac{1}{\left(1-q^{\binom{5}{2}}x^4\right)}
    \cdot
    \cdots
    \\
    =&
    \left(1+q^{\binom{2}{2}}x+q^{2{\binom{2}{2}}}x^2+q^{3{\binom{2}{2}}}x^3+q^{4{\binom{2}{2}}}x^4+\cdots \right)\cdot
    \\
    &
    \left(1+q^{\binom{3}{2}}x^2+q^{2{\binom{3}{2}}}x^4+q^{3{\binom{3}{2}}}x^6+q^{4{\binom{3}{2}}}x^8+\cdots \right)\cdot
    \\
    &
    \left(1+q^{{\binom{4}{2}}}x^3+q^{2{\binom{4}{2}}}x^6+q^{3{\binom{4}{2}}}x^9+q^{4{\binom{4}{2}}}x^{12}+\cdots \right)\cdot
    \\
    &
    \left(1+q^{{\binom{5}{2}}}x^4+q^{2{\binom{5}{2}}}x^8+q^{3{\binom{5}{2}}}x^{12}+q^{4{\binom{5}{2}}}x^{16}+\cdots \right)\cdot\cdots.
\end{align*}

If we distribute and simplify, for example, the coefficient of $x^4$, we see that it is
\[ q^{4{\binom{2}{2}}}+q^{2{\binom{3}{2}}}+q^{{\binom{2}{2}}+{\binom{4}{2}}}+q^{2{\binom{2}{2}}+{\binom{3}{2}}}+q^{{\binom{5}{2}}},  \]
where the various terms correspond to the partitions
\[(1,1,1,1),\,(2,2),\,(3,1),\,(2,1,1),\,\textrm{and}\,(4),\]
respectively, by 
\[
    (\lambda_1,\lambda_2,\ldots,\lambda_l)\mapsto q^{\left(\sum_{i=1}^{l}{\binom{\lambda_i+1}{2}}\right)}.
\]
This is true, in general, for the coefficient of $x^n$, for all positive integers $n$. To see this, consider the coefficient on $x^n$ in the power series expansion of formula~\ref{eq1}. If we set $q=1$ in the product of formula~\ref{eq1}, we obtain the generating function of $P(n)$ (the number of partitions of $n$). Hence there are $P(n)$ different ways to obtain a power of $x^n$. So the $x^n$ term in formula~\ref{eq1} will be of the form
\begin{equation*}
    \sum_{j=1}^{P(n)}\prod _{i=1}^{m_j}q^{\left(a_{j,i}{\binom{\lambda_{j,i}+1}{2}}\right)}x^{\left(a_{j,i}\lambda_{j,i}\right)},
\end{equation*}
where $a_{j,i},\lambda_{j,i}>0$, and $\sum_{i=1}^{m_j}a_{j,i}\lambda_{j,i}=n$. Thus the coefficient on $x^n$ will be
\begin{equation}
    \sum_{j=1}^{P(n)}q^{\left(\sum_{i=1}^{m_j}{ a_{j,i}{\binom{\lambda_{j,i}+1}{2}} }\right) }.
    \label{eq2}
\end{equation}
Since each ${\binom{\lambda_{j,i}+1}{2}}$ in formula~\ref{eq2} comes from a different term of the product in formula~\ref{eq1}, we have that $\lambda_{j,i}\neq\lambda_{j,k}$ whenever $i\neq k$. Therefore, by reordering, we may choose the $\lambda_{j,i}$ such that $\lambda_{j,i}>\lambda_{j,{i+1}}> 0$, for $1\leq i < m_j$. It follows that $(\lambda_{j,1}^{a_{j,1}},\lambda_{j,2}^{a_{j,2}},\ldots,\lambda_{j,{m_j}}^{a_{j,{m_j}}})$ is a partition of $n$, where $\lambda_{j,i}$ is repeated $a_{j,i}$ times.

Since every partition $(\lambda_{j,1}^{a_{j,1}},\lambda_{j,2}^{a_{j,2}},\ldots,\lambda_{j,{m_j}}^{a_{j,{m_j}}})$ of $n$ contributes to the coefficient of $x^n$ in formula~\ref{eq2}, the formula may be rewritten as 

\begin{equation*}
 \sum_{\lambda\vdash n} q^{\left(\sum_{i=1}^{n}\binom{\lambda_i+1}{2}\right)}.
\end{equation*}
By Lemma~\ref{lemma1}, $e_2(\lambda)=\binom{n+1}{2}-\sum_{i=1}^{n}\binom{\lambda_i+1}{2}$, thus 
\begin{equation*}
    \sum_{\lambda\vdash n} q^{\left(\sum_{i=1}^{n}\binom{\lambda_i+1}{2}\right)}=\sum_{\lambda\vdash n} q^{\left(\sum_{i=1}^{n} \binom{n+1}{2}-e_2(\lambda)\right)}.
\end{equation*}
Finally, by Proposition~\ref{prop1}, we have that the general coefficient on $x^n$ in the power series expansion of formula~\ref{eq1} is
\begin{equation*}
    \sum_{\lambda\vdash n} q^{\left(\sum_{i=1}^{n} \binom{n+1}{2}-g(\lambda)\right)}.
\end{equation*}
\end{proof}
\newpage
\bibliographystyle{Bibtex.sty}
\bibliography{References}

\bigskip
\hrule
\bigskip

\noindent 2010 {\it Mathematics Subject Classification}:
Primary 05A17, Secondary 05A15, 05E05, 06A06, 91B99.

\noindent \emph{Keywords: }
Gini index, integer partition, dominance order, antichain, generating function.

\bigskip
\hrule
\bigskip

\noindent (Concerned with sequences
\seqnum{A076269}, and 
\seqnum{A337206}.)

\bigskip
\hrule
\bigskip
\end{document}